\documentclass[12pt]{article}

\usepackage{amsfonts,amsmath,amsxtra,amsthm}
\input xy
\xyoption{all}
\usepackage[all,knot]{xy}

\newtheorem{theorem}{Theorem}[section]
\newtheorem*{theorema}{Theorem}
\newtheorem{lemma}[theorem]{Lemma}

\theoremstyle{definition}
\newtheorem{definition}[theorem]{Definition}
\newtheorem{example}[theorem]{Example}

\newtheorem{corollary}[theorem]{Corollary}

\theoremstyle{remark}
\newtheorem{remark}[theorem]{Remark}

\theoremstyle{proposition}
\newtheorem{proposition}[theorem]{Proposition}

\theoremstyle{propositions}

\numberwithin{equation}{section}

\title{Arc spaces and DAHA representations}
\author{E. Gorsky}
\date{}

\begin{document}
\maketitle

\begin{abstract}
A theorem of Y. Berest, P. Etingof and V. Ginzburg states that finite dimensional irreducible representations
of a type A rational Cherednik algebra are classified by one rational number $m/n$.
Every such representation is a representation of the symmetric group $S_n$. 
We compare certain multiplicity spaces in its decomposition into irreducible representations of $S_n$ with the
spaces of differential forms on a zero-dimensional moduli space associated with the plane curve singularity $x^m=y^n$.
\end{abstract}

\section{Introduction}

Rational double affine Hecke algebras (DAHA) were introduced by I. Cherednik (\cite{chered}) in his study of the Macdonald conjectures.
Their representation theory was extensively studied by C. Dunkl (\cite{dunkl1},\cite{dunkl2}), Y. Berest, P. Etingof and V. Ginzburg (\cite{beg1},\cite{beg2}), I. Gordon and T. Stafford (\cite{gordon},\cite{gosta1},\cite{gosta2}), M. Varagnolo and E. Vasserot (\cite{vava}). Their
relation to the geometry of the Hilbert schemes of points and affine Springer fibers was discussed in \cite{gosta2},\cite{vava},\cite{yun}.
We refer the reader to the lectures \cite{etingof} and the references therein for more complete bibliography.

By construction, rational DAHA $H_{n,c}$ of type $A_{n-1}$ with parameter $c$ has a representation $M_{c}$ by Dunkl operators in the space of polynomials on the Cartan subalgebra $V_n$. In \cite{beg1} it was shown that $H_{n,c}$  has a finite-dimensional representation if and only if $c=m/n,\quad m\in (m,n)=1$. In this case there exists a unique irreducible finite-dimensional representation $L_{m/n}$, which can be constructed as a quotient of $M_{m/n}$ by a certain ideal $I_{m/n}$.

By construction, $L_{m/n}$ carry a  natural representation of the symmetric group $S_n$, so we can split it into irreducible representations
of $S_n$. The symmetric and antisymmetric parts of $L_{m/n}$ were extensively studied in connection to the representation theory of the spherical DAHA (\cite{beg1},\cite{vava}) and the geometry of the Hilbert scheme of points (\cite{gordon},\cite{gosta1},\cite{gosta2},\cite{haiman}). In particular, 
the space $L_{n+1\over n}$ is related to the $q,t$-Catalan numbers introduced by A. Garsia and M. Haiman (\cite{gaha}).
We are interested in a slightly more general problem of describing the multiplicities of the exterior power $\Lambda^{k}V_n$ in 
$L_{m/n}$. Since $\Lambda^{k}V_n$ is known to be an irreducible representation of $S_n$, we can describe this multiplicity space as
$$\mbox{\rm Hom}_{S_n}(\Lambda^{k}V_n,L_{m/n}).$$ This space is conjectured to be related to some homological invariants of torus knots (\cite{os},\cite{ors},\cite{gors}). 
Since $$\mbox{\rm Hom}_{S_n}(\Lambda^{k}V_n,M_{m/n})\simeq \Omega^{k}(V_n//S^{n}),$$
one can ask if the ideal $I_{m/n}$ is related to some natural ideal in  $\Omega^{k}(V_n//S^{n}).$

In \cite{gfs} L. Goettsche, B. Fantechi and D. van Straten constructed a zero-dimensional moduli space
$\mathcal{M}_{m,n}$ defined by the coefficients in $z$-expansion of the equation
$$(1+z^2u_2+z^3u_3+\ldots+z^nu_n)^{m}=(1+z^2v_2+z^3v_3+\ldots+z^{m}v_{m})^{n}.$$
Our main result is the following

\begin{theorema}
\label{main}
The following isomorphism holds:
$$\mbox{\rm Hom}_{S_n}(\Lambda^{k}V_n,L_{m/n})\simeq \Omega^{k}(\mathcal{M}_{m,n}).$$
\end{theorema}

The proof is explicit: we identify $\mathcal{M}_{m,n}$ with a subscheme in $\mbox{\rm Spec}~\mathbb{C}[u_2,\ldots,u_n]$,
and identify $u_k$ with the $k$-th elementary symmetric polynomial on $V_n$.
This allows us to embed $\mathcal{M}_{m,n}$ into the quotient $V_n//S^{n}$.
It rests to compare the defining ideals in both cases.

\begin{corollary}
The left hand side is symmetric in $m$ and $n$:
$$\mbox{\rm Hom}_{S_n}(\Lambda^{k}V_n,L_{m/n})=\mbox{\rm Hom}_{S_m}(\Lambda^{k}V_{m},L_{n/m}).$$
\end{corollary}

\begin{remark}
This relation was proved for $k=0$ by D. Calaque, B. Enriquez and P. Etingof in \cite{calaque} by different method.
I. Losev announced (\cite{losev}) a generalization of their proof for higher values of $k$.
\end{remark}

In sections 2 and 3 we briefly discuss the constructions of the representation $L_{m/n}$ and the moduli space $\mathcal{M}_{m,n}$.
In section 4 we compare the constructions and prove the main theorem. In section 5 we discuss the action of Dunkl operators on
$\Omega^{k}(\mathcal{M}_{m,n})$ and the action of the Olshanetsky-Perelomov Hamiltonians on $u_l$ and $v_l$.

The author is grateful to J. Rasmussen, A. Oblomkov, V. Shende and A. Kirillov Jr. for useful discussions.
This research was partially supported by the grants RFBR-10-01-00678, NSh-8462.2010.1  and the Dynasty fellowship for young scientists.

\section{Rational Cherednik algebras}

\begin{definition}(\cite{beg1})
The rational Cherednik algebra $H_{c}$ of type $A_{n-1}$ with parameter $c$ is an associative algebra 
generated by $V=\mathbb{C}^{n-1}, V^{*}$ and $S_n$ with the following defining relations:
\begin{eqnarray}
\label{DAHA}
\sigma\cdot x\cdot \sigma^{-1}=\sigma(x),\quad \sigma\cdot y\cdot \sigma^{-1}=\sigma(y),\quad \forall x\in V, y\in V^{*}, \sigma\in S_n\nonumber \\
x_1\cdot x_2=x_2\cdot x_1,\quad y_1\cdot y_2=y_2\cdot y_1\quad \forall x_1,x_2\in V,y_1,y_2\in V^{*}\\
y\cdot x-x\cdot y=<y,x>-c\sum_{s\in \mathcal{S}}<\alpha_s,x><y,\alpha_{s}^{\vee}>\cdot s\quad \forall x\in V,y\in V^{*},\nonumber
\end{eqnarray}
where $\mathcal{S}\subset S_n$ is the set of all transpositions, and $\alpha_s, \alpha_{s}^{\vee}$ are the corresponding roots and coroots.
\end{definition}

The last defining relation is motivated by the following construction of C. Dunkl (\cite{dunkl1}). 

\begin{definition}
We introduce the Dunkl operators with parameter $c$ by the formula
$$D_{i}={\partial\over \partial x_i}-c\sum_{j\neq i}{s_{ij}-1\over x_i-x_j}.$$
\end{definition}

\begin{proposition}
Consider the space $\mathbb{C}[V]$ of polynomial functions on $V$, where the elements of $V$ act by multiplication and the basis of $V^{*}$ acts by Dunkl operators. This produces a representation of $H_{c}$, i. e. all defining relations (\ref{DAHA}) hold. 
This representation is denoted by $M_{c}$.
\end{proposition}
 
An theorem of Y. Berest, P. Etingof, and V. Ginzburg says that for $c > 0$, all finite dimensional
irreducible representations of $H_c$ can be obtained from this construction:

\begin{theorem}(\cite{beg1})
$H_c$ has finite dimensional representations if and only if $c = m/n$, where $m$
is an integer  and $(m, n)=1$. In this case, $H_{m/n}$ has a unique (up to isomorphism) finite
dimensional irreducible representation $L_{m/n}$. For $c = m/n > 0$, $L_c = M_c/I_c$, where $I_c$ is an ideal generated by
some homogeneous polynomials of degree $m.$
\end{theorem}

Following \cite{dunkl2} and \cite{chmutova}, we would like to give an explicit construction of these polynomials for $c=m/n$:
$$f_{i}=\mbox{\rm Coef}_{m}[(1-zx_i)^{-1}\prod_{i=1}^{n}(1-zx_i)^{m\over n}].$$

\begin{remark}
It is known that $f_i$ form a regular sequence, hence $I_{m/n}$ defines a 0-dimensional complete intersection in $V$. 
The Bezout's theorem implies the dimension formula (\cite{beg1})
$$\dim L_{m/n}=\dim \mathbb{C}[V]/I_{m/n}=m^{n-1}.$$
\end{remark}

Let us explain the choice of the polynomials $f_i$. From now on we will assume that $c=m/n$.
\begin{definition}
Following \cite{dunkl2}, let us introduce the formal series
$$F(z)=\prod_{i=1}^{n}(1-zx_i)^{m/n},\quad B_{i}(z)={1\over 1-zx_i}F(z).$$
\end{definition}



\begin{lemma}(\cite{dunkl2}) 
The action of the Dunkl operators on $B_i(z)$ is given by the formula
$$D_{s}B_{i}(z)=\delta_{is}\left(z^2{dB_i\over dz}+(1-m)zB_i(z)\right)$$
\end{lemma}



\begin{corollary}(\cite{dunkl2})
\label{db}
$$D_{s}\mbox{\rm Coef}_{k}[B_i(z)]=\delta_{is}(k-m)\mbox{\rm Coef}_{k-1}[B_i(z)].$$
In particular,
$$D_s(f_i)=D_{s}\mbox{\rm Coef}_{m}[B_i(z)]=0\quad.$$
This explains why $f_i$ generate an ideal invariant under DAHA action.
\end{corollary}

\section{Arc space on a singular curve}

In \cite{gfs} L. Goettsche, B. Fantechi and D. van Straten constructed a zero-dimensional 
quasihomogeneous complete intersection associated with a curve $\{x^m=y^n\}$. 
As before, we will assume that $(m,n)=1$.

Consider a plane curve singularity $C=\{x^{m}=y^{n}\}$, and its uniformization 
$(x,y)=(t^{n},t^{m}).$  Let us consider a general deformation of this parametrization:
$$(x(t),y(t))=(t^{n}+u_{1}t^{n-1}+\ldots+u_{n},t^{m}+v_{1}t^{m-1}+\ldots+v_{n}).$$

Consider the ideal $I_{m,n}$ generated by the coefficients in $t$-expansion of the equation
\begin{equation}
\label{eqmn}
(t^{n}+u_{2}t^{n-2}+\ldots+u_{n})^{m}-(t^{m}+v_{2}t^{m-2}+\ldots+v_{m})^{n}=0.
\end{equation}

\begin{remark}
By shifting the parameter $t$ we can annihilate the coefficient $u_{1}$.
Since
$$x(t)^{m}-y(t)^{n}=(mu_{1}-nv_{1})t^{mn-1}+\mbox{\rm terms of lower degree},$$
the equation $u_{1}=0$ implies $v_{1}=0$.
\end{remark}

\begin{definition}
 The moduli scheme of arcs on $C$ is defined as
$$\mathcal{M}_{C}=\mbox{\rm Spec}\quad \mathbb{C}[u_{2},\ldots,u_{n},v_{2},\ldots v_{m}]/I_{m,n}.$$
\end{definition}

\begin{lemma}(\cite{gfs}, Example 1)
The scheme $\mathcal{M}_{C}$ is a zero-dimensional complete intersection.
\end{lemma}

\begin{proof}
The equation $x(t)^{m}-y(t)^{n}=0$ is equivalent to the equation
\begin{equation}
\label{logder}
mx'(t)y(t)-ny'(t)x(t)=0.
\end{equation}
The left hand side of (\ref{logder}) is a polynomial in $t$ of degree $m+n-3$. 
Therefore the ideal $I_{m,n}$ is generated by a sequence of the $(m+n-2)$ equations on  $(m-1)+(n-1)=(m+n-2)$ variables.
Since $m$ and $n$ are coprime, the reduced scheme consists of one point $(x(t),y(t))=(t^n,t^m).$
\end{proof}

\begin{corollary}
The algebra of algebraic differential forms on this moduli space can be described as
$$\Omega^{\bullet}(\mathcal{M}_{C})=\Omega^{\bullet}(\mathbb{C}^{m+n-2})/(\phi\cdot \omega_1+d\phi\wedge \omega_2|\phi\in I_{m,n}).$$
\end{corollary}

\begin{definition}
We assign the $q$-grading to $u_i$ and $v_i$ by the formula
$$q(u_i)=q(v_i)=i.$$
\end{definition}

\begin{lemma}(\cite{gfs}) The multiplicity of $\mathcal{M}_{C}$ is given by the formula
\label{mult0}
$$\mbox{\rm mult}(\mathcal{M}_{C})={(m+n-1)!\over m!n!}$$
\end{lemma}

\begin{proof}
One can check that the ideal $I_{m,n}$ is weighted homogeneous with respect to the $q$-grading,
and the multiplicity of $\mathcal{M}_{C}$ can be computed  using Bezout's theorem:
$$\mbox{\rm mult}(\mathcal{M}_{C})={2\cdot 3\cdot \ldots\cdot (m+n-1)\over 2\cdot 3 \ldots \cdot n\cdot 2\cdot 3 \ldots \cdot m}={(m+n-1)!\over m!n!}.$$
\end{proof}

\begin{lemma}(\cite{gfs})
The Hilbert series of $\mathbb{C}[u_2,\ldots,u_n,v_2,\ldots,v_n]/I_{m,n}$ with respect to the $q$-grading equals to
$$H_{m,n}(q)={[(m+n-1)!]_{q}\over [m!]_{q}[n!]_{q}}=\prod_{k=2}^{n}{(1-q^{m+k-1})\over (1-q^{k})}.$$
\end{lemma}

\begin{proof}
The ideal $I_{m,n}$ is generated by the weighted homogeneous regular sequence of equations of weights 
$2,3,\ldots (m+n-1)$, so the proof is analogous to the Lemma \ref{mult0}.
\end{proof}


For the further discussions we have to slightly change the notations.
Let us change $t$ to $z=t^{-1}$, then the equation (\ref{eqmn}) will have a form
\begin{equation}
\label{equv}
(1+z^2u_2+\ldots+z^{n}u_n)^{m}=(1+z^2v_2+\ldots+z^{m}v_{m})^{n}.
\end{equation}

\begin{definition}
Let $J_{m/n}$ denote the ideal in $\mathbb{C}[u_1,\ldots, u_{n}]$ generated by 
the coefficients of the series $(1+z^2u_2+\ldots+z^{n}u_n)^{m\over n}$, starting from $m+1$-st.
\end{definition}

\begin{lemma}
\label{utov}
Using (\ref{equv}), one can express $v_i$ through $u_j$.
The remaining equations on $u_{i}$ generate the ideal $J_{m/n}$.
\end{lemma} 

\begin{proof}
Let us take the $n$th root of both parts of (\ref{equv}):
$$1+z^2v_2+\ldots+z^{m}v_{m}=(1+z^2u_2+\ldots+z^{n}u_n)^{m\over n}.$$
This allows us to express $v_i$ through $u_j$ explicitly, and the remaining equations on $u_j$ express the fact
that the right hand side should be a polynomial of degree at most m.
\end{proof}

\begin{corollary}
$$\mathbb{C}[u_1,\ldots,u_n]/J_{m/n}\simeq \mathbb{C}[v_1,\ldots,v_m]/J_{n/m}.$$
\end{corollary}

\begin{example}
\label{ex2k}
For $n=2$ one can check that $J_{(2k+1)/2}$ is generated by a single polynomial $u_2^{k+1}$.
Therefore the algebra of differential forms $\Omega^{\bullet}(\mathcal{M}_{C})$ is defined by the equations
$$u_2^{k+1}=0,\quad u_2^{k}du_2=0,$$
and its basis consists of forms
$$1,u_2,\ldots,u_2^{k},du_2,u_2du_2,\ldots,u_2^{k-1}du_2.$$
\end{example}

\begin{example}
\label{ex34}
Let $n=3,m=4$. The ideal $J_{4/3}$ is generated by the coefficients of the series
$(1+u_2z^2+u_3z^3)^{4/3}$, starting from 5-th, hence its generators are
$${4\over 9}u_2u_3,\quad {2\over 9}u_3^2-{4\over 81}u_2^3.$$
The basis in the quotient is presented by
$1, u_2, u_2^2, u_2^3, u_3.$

The algebra $\Omega^{\bullet}(\mathcal{M}_{C})$ is defined by two more equations
$$u_2du_3+u_3du_2=0,\quad 2u_3du_3-{2\over 3}u_2^2du_2=0.$$
Therefore we have 5 one-forms 
$$\Omega^{1}(\mathcal{M}_{C})=<du_2, du_3, u_2du_2, u_2du_3, u_2^2du_2>$$
and one two-form $du_2du_3.$
\end{example}

\section{A comparison}

It turns out that the ideal $I_{m/n}$ constructed in \cite{dunkl2} is related to the moduli space $\mathcal{M}_{C}$.
Recall that $V$ denotes the standard $(n-1)$-dimensional representation of the symmetric group $S_n$ with the coordinates $x_i$
modulo the relation $\sum x_i=0$.

\begin{definition}
Let us introduce the elementary symmetric polynomials $u_i(x_1,\ldots,x_n)\in \mathbb{C}[V]^{S_n}$ by the formula
$$U(z)=\prod_{i=1}^{n}(1-zx_i)=1+\sum_{i=2}^{n}z^{i}u_{i}(x).$$
\end{definition}

\begin{lemma}
\label{newton}
Let us introduce the power sums
$$p_i=x_1^{i}+\ldots+x_n^{i}.$$
Then
$${d\over dz}\ln U(z)=-\sum_{i=0}^{\infty}p_{i+1}z^{i}.$$
\end{lemma}


\begin{theorem}
\label{dahalow}
For $c=m/n$ one have
$$(I_{c})^{S_n}=J_{m/n}.$$
\end{theorem}

\begin{proof}
By construction, the set of generators of the ideal $I_{m/n}$ form a standard representation of $S_n$ (cf. Remark \ref{sum0}). 
Therefore to construct the generators of the symmetric part of $I_{m/n}$, we have to compute
$$[V\otimes \mathbb{C}[V]]^{S_n}/\mathbb{C}[V]^{S_n}_{+}=\mbox{\rm Hom}_{S_n}(V,\mathbb{C}[V]/\mathbb{C}[V]^{S_n}_{+}),$$
where $\mathbb{C}[V]^{S_n}_{+}$ denotes the ideal generated by the symmetric polynomials of positive degree.
It is well known that $\mathbb{C}[V]/\mathbb{C}[V]^{S_n}_{+}$  is isomorphic to the regular representation of $S_n$, so
there are $(n-1)$ different copies of $V$ in it, generated by $x_i^{k}$ for $1\le k\le n-1$. This means that the 
symmetric part of the ideal $I_{m/n}$ is generated by $(n-1)$ polynomials:
$$(I_{c})^{S_n}=<\sum_{i}x_i^{k}f_i|1\le k\le n-1>.$$
Let us compute these generators.

Recall that $$F(z)=\prod_{i=1}^{n}(1-zx_i)^{m/n}=\sum_{k=0}^{\infty} v_{k}(x)z^{k}.$$
By Lemma \ref{newton}
$${d\over dz}F(z)={m\over n}F(z){d\over dz}\ln U(z)=-{m\over n}F(z)\cdot \sum_{i=0}^{\infty}p_{i+1}z^{i},$$
Therefore
\begin{equation}
\label{recursion}
(k+1)v_{k+1}=-{m\over n}\sum_{i=0}^{k}v_{k-i}p_{i+1}.
\end{equation}

On the other hand,
$$\sum_{i}f_{i}x_i^{k}=\sum_{i}\mbox{\rm Coef}_{m}{x_i^{k}\over 1-zx_i}F(z)=\sum_{j=0}^{m}v_{m-j}p_{j+k}.$$
Therefore by (\ref{recursion})
$$\sum_{i}f_{i}x_i=\sum_{j=0}^{m}v_{m-j}p_{j+1}=-{n\over m}(m+1)v_{m+1},$$
$$\sum_{i}f_{i}x_i^2=\sum_{j=0}^{m}v_{m-j}p_{j+2}=-{n\over m}(m+2)v_{m+2}-v_{m+1}p_1,\ldots$$
$$\sum_{i}f_{i}x_i^k=\sum_{j=0}^{m}v_{m-j}p_{j+k}=-{n\over m}(m+k)v_{m+k}-\sum_{j=1}^{k-1}v_{m+j}p_{k-j},$$
and $v_{m+1},\ldots, v_{m+n-1}$ can be obtained from $\sum_{i}f_{i}x_i^k$ with a triangular change of variables.

It rests to note that by Lemma \ref{utov} the polynomials $v_{m+1},\ldots, v_{m+n-1}$ generate the ideal $J_{m/n}$.




\end{proof}

\begin{remark}
\label{sum0}
If we plug in $k=m-1$ in (\ref{recursion}), we will get\\
$mv_{m}=-{m\over n}\sum_{i=0}^{m-1}v_{k-i}p_{i+1},$
hence
$nv_{m}+\sum_{i=0}^{m-1}v_{k-i}p_{i+1}=0,$
and
$$\sum_{j}f_j=\sum_{j}\mbox{\rm Coef}_{m}[(1-zx_j)^{-1}F(z)]=nv_{m}+\sum_{i=1}^{m}v_{m-i}p_{i}=0.$$
\end{remark}

\begin{corollary}
\label{nextcoef}
$$\mbox{\rm Coef}_{k}[B_i(z)]\in I_{m/n}\quad \mbox{\rm for}\quad k\ge m.$$
\end{corollary}

\begin{proof}
We have
$$\mbox{\rm Coef}_{k}[B_i(z)]=\mbox{\rm Coef}_{k}[{F(z)\over 1-zx_i}]=\sum_{a=0}^{k}x_i^{a}v_{k-a}=$$
$$\sum_{a=0}^{k-m-1}x_i^{a}v_{k-a}+\sum_{a=k-m}^{k}x_i^{a}v_{k-a}.$$
By Theorem \ref{dahalow} the first sum belongs to $I_{m/n}$, and the second sum can be rewritten as
$$\sum_{a=0}^{m}x_i^{k-m+a}v_{m-a}=x_i^{k-m}\sum_{a=0}^{m}x_i^{a}v_{m-a}=x_i^{k-m}\mbox{\rm Coef}_{m}[B_i(z)]\in I_{m/n}.$$
\end{proof}

\begin{lemma}
\label{exc}
Let $h(x_1,\ldots,x_n)$ be a symmetric function in $x_2,\ldots, x_n$.
There exist functions $\phi_1,\ldots,\phi_{n-1},\rho_1,\ldots,\rho_{n-1}\in \mathbb{C}[V]^{S_n}$ such that
\begin{equation}
\label{decomp}
h\cdot f_1=\sum_{k>m}(\phi_{j}{\partial v_{k}\over \partial x_1}+\rho_{j}v_{k}).
\end{equation}
\end{lemma}

\begin{proof}
One can present $h$ as a linear combination of some powers of $x_1$ multiplied by some symmetric polynomials in $x_1,\ldots,x_n$.
Therefore it is sufficient to prove (\ref{decomp}) for $h=x_1^{k}$.

Since $$(1-zx_1){\partial F(z)\over \partial x_1}=-{m\over n}zF(z),$$ one has
$$x_1{\partial v_k\over \partial x_1}={\partial v_{k+1}\over \partial x_1}+{m\over n}v_{k}.$$
Using this equation, one can express $$x_1^{k}f_1=-{n\over m}x_1^{k}{\partial v_{m+1}\over \partial x_1}$$
via $v_k$ and ${\partial v_{k}\over \partial x_1}$ with $k>m$.
\end{proof}

\begin{theorem}
\label{t1}
$$\mbox{\rm Hom}_{S_n}(\Lambda^{k}V,L_{m/n})\simeq \Omega^{k}(\mathcal{M}_{m,n}).$$
\end{theorem}

Remark that 
$$\mbox{\rm Hom}_{S_n}(\Lambda^{k}V,M_{m/n})=\mbox{\rm Hom}_{S_n}(\Lambda^{k}V,\mathbb{C}[V])=\Omega^{k}(V)^{S_n}\simeq \Omega^{k}(V//S_n),$$
where the last isomorphism follows  from the results of \cite{solomon}: every $S_n$-invariant differential form on $V$ can be obtained as a pullback of a form on $V//S_n$. To fix the notation, we give the following

\begin{definition}
We introduce a map
$$\lambda: \Omega^{k}(V//S_n)\rightarrow \mbox{\rm Hom}_{S_n}(\Lambda^{k}V,\mathbb{C}[V])$$
by the formula
$$\lambda_{i_1,\ldots,i_k}(\omega)=\pi^{*}\omega({\partial \over \partial x_{i_1}},\ldots,{\partial \over \partial x_{i_k}}),$$
where $\pi:V\rightarrow V//S_n$ denotes the natural projection.
\end{definition}

\begin{proposition}
The following equation holds:
$$\lambda_{i_1,\ldots,i_s}(du_{k_1}\wedge \ldots \wedge du_{k_s})=\left|{\partial u_{k_a}\over \partial x_{i_b}}\right|.$$
\end{proposition}

{\bf Proof of Theorem \ref{t1}}. Let us check that the map $\lambda$ sends the defining equations of $\Omega^{\bullet}(\mathcal{M}_C)$ inside the ideal $I_{m/n}$.

Recall that
$$F(z)=\prod_{i=1}^{n}(1-zx_i)^{m/n}=\sum_{k=0}^{\infty} v_{k}(u)z^{k},$$
and the  defining ideal of $\Omega^{\bullet}(\mathcal{M}_C)$ is generated by the equations  $$v_k(u)=dv_k(u)=0\quad \mbox{\rm for}\quad k>m$$
We checked in Theorem \ref{dahalow} that $\lambda(v_k)\in I_{m/n}$, let us check that
$\lambda(dv_k)\in I_{m/n}.$

Remark that
$$\lambda_i(dv_k)={\partial v_k\over \partial x_i}=
\mbox{\rm Coef}_{k}{\partial F(z)\over \partial x_i} =-{m\over n}\mbox{\rm Coef}_{k-1}[{F(z)\over 1-zx_i}].$$
This coefficient belongs to $I_{m/n}$ by Corollary \ref{nextcoef}.

Similarly to the proof of Theorem \ref{dahalow}
one can check that every element of $\mbox{\rm Hom}_{S_n}(\Lambda^{k}V,I_{m/n})$ can be presented as a combination of the determinants
of the form
$$\left|
\begin{matrix}
f_{\alpha_1}h_1 & f_{\alpha_2}h_2 & \ldots & f_{\alpha_k}h_k\\
{\partial u_{\beta_1}\over \partial x_{\alpha_1}} & {\partial u_{\beta_1}\over \partial x_{\alpha_2}} & \ldots & {\partial u_{\beta_1}\over \partial x_{\alpha_k}}\\
\vdots & \vdots  & \ddots & \vdots\\
{\partial u_{\beta_{k-1}}\over \partial x_{\alpha_1}} & {\partial u_{\beta_{k-1}}\over \partial x_{\alpha_2}} & \ldots & {\partial u_{\beta_{k-1}}\over \partial x_{\alpha_k}}\\
\end{matrix}
\right|$$
with symmetric coefficients, where $h_i$ are symmetric in all variables but $x_{\alpha_i}$.
By Lemma \ref{exc}, we can present (modulo $J_{m/n}$) every such determinant as a combination of the expressions
$$\lambda_{\alpha_1,\ldots,\alpha_k}(dv_{s}\wedge du_{\beta_1}\wedge \ldots \wedge du_{\beta_{k-1}}),\quad s>m $$
with symmetric coefficients. It rests to note that the form $$dv_{s}\wedge du_{\beta_1}\wedge \ldots \wedge du_{\beta_{k-1}}=dv_{s}\wedge \omega$$
belongs to the defining ideal of $\Omega^{\bullet}(\mathcal{M}_C)$.
$\square$

\section{Action of Dunkl operators}

We start with the following reformulation of the Corollary \ref{db}.

\begin{lemma}
\label{Dv}
$$D_{s}\lambda_{i}(dv_k)=\delta_{is}(k-1-m)\lambda_{i}(dv_{k-1}).$$
\end{lemma}

\begin{proof}
Remark that 
$$\lambda_{i}(dv_k)={\partial v_k\over \partial x_i}=-{m\over n}\mbox{\rm Coef}_{k-1}[B_i(z)],$$
by Corollary \ref{db} 
$$D_{s}\mbox{\rm Coef}_{k-1}[B_i(z)]=\delta_{is}(k-1-m)\mbox{\rm Coef}_{k-2}[B_i(z)].$$
hence
$$D_{s}\lambda_{i}(dv_k)=\delta_{is}(k-1-m)\lambda_{i}(dv_{k-1}).$$
\end{proof}

\begin{lemma}(\cite{dunkl2})
The following product rule holds for Dunkl operators:
$$D_{i}(fg)=D_i(f)g+D_i(g)f+{m\over n}\sum_{i\neq j}{(f-s_{ij}(f))(g-s_{ij}(g))\over x_i-x_j}.$$
\end{lemma}

\begin{corollary}
\label{symleib}
If $g$ is a symmetric polynomial then
$$D_{i}(fg)=D_i(f)g+D_i(g)f.$$
\end{corollary}

\begin{lemma}(\cite{dunkl2})
Suppose $f_1,\ldots,f_m$ are polynomials satisfying $(ij)f_{l}=f_{l}$ if $l\notin \{i,j\}.$
Then
\begin{equation}
\label{dunklprod}
D_{i}(f_1\cdots f_m)=\sum_{l=1}^{m}(D_{i}f_{l})\prod_{s\neq l}f_{s}+{m\over n}\sum_{l\neq i}{(f_i-s_{il}(f_i))(f_l-s_{il}(f_{l}))\over x_i-x_l}\prod_{s\neq l,i}f_{s}.
\end{equation}
\end{lemma}

\begin{lemma}
\label{dunkldet}
Suppose that the functions $a_j, 1\le j\le k$ are symmetric with respect to all variables but $x_1$.
Consider a matrix $M=(s_{1i}(a_j))_{j=1}^{k}$.
Then 
\begin{equation}
D_{i}\det(M)=\sum_{l}\det(M_{i,l}),
\end{equation}
where $M_{s,l}$ denotes the matrix $M$ where the entries in the $l$'th row are replaced by their images under $D_i$.
\end{lemma}

\begin{remark}
This lemma shows that although $D_{i}$ is not a first order differential operator, it acts on these determinants as a first order differential operator would act.
\end{remark}

\begin{proof}
Let us expand $\det(M)$ and apply the equation (\ref{dunklprod}). We have to show that the "correction terms" with divided differences will cancel out.
These terms are labelled by the pairs $(\sigma,l)$ where $l\neq i$ and $\sigma\in S_k$, and the terms corresponding to $(\sigma,l)$ and $((il)\sigma,l)$ have
opposite sign but same value
$${1\over x_i-x_l}(s_{1i}a_{\sigma(i)}-s_{il}s_{1i}a_{\sigma(i)})(s_{1l}a_{\sigma(l)}-s_{il}s_{1l}a_{\sigma(l)}))=$$ $${1\over x_i-x_l}(s_{1i}a_{\sigma(i)}-s_{1l}a_{\sigma(i)})(s_{1l}a_{\sigma(l)}-s_{1i}a_{\sigma(l)})).$$
\end{proof}

We are ready to describe the action of Dunkl operators on the image of the map $\lambda$. By Lemma \ref{symleib} 
and it is sufficient to compute the action of $D_i$ on the components of the differential form $dv_{\alpha_1}\wedge\ldots\wedge dv_{\alpha_k}$.

\begin{theorem}
Suppose that $\beta_1<\ldots<\beta_j$. If $\beta_1>1$, then
$$D_1\left|{\partial v_{\alpha_i}\over \partial x_{\beta_j}}\right|=0.$$
If $\beta_1=1$, then
$$D_1\left|{\partial v_{\alpha_i}\over \partial x_{\beta_j}}\right|=\left|\begin{matrix} 
(\alpha_1-1-m){\partial v_{\alpha_1-1}\over \partial x_{1}} & \ldots & (\alpha_k-1-m){\partial v_{\alpha_k-1}\over \partial x_{1}}\\
{\partial v_{\alpha_1}\over \partial x_{\beta_2}} & \ldots & {\partial v_{\alpha_k}\over \partial x_{\beta_2}} \\
\vdots & \ddots & \vdots\\
{\partial v_{\alpha_1}\over \partial x_{\beta_k}} & \ldots & {\partial v_{\alpha_k}\over \partial x_{\beta_k}} \\
\end{matrix}\right|$$
\end{theorem}

\begin{proof}
Follows from Lemma \ref{dunkldet} and Lemma \ref{Dv}.
\end{proof}

\begin{definition}(\cite{op},\cite{etingof})
The quantum Olshanetsky-Perelomov Hamiltonians are defined as
$$H_{k}=\sum_{s=1}^{n}D_{s}^{k}.$$
\end{definition}

\begin{lemma}
\label{sumlam}
$$\sum_{i}\lambda_{i}(dv_k)=(k-1-m)v_{k-1}$$
\end{lemma}

\begin{proof}
$$\sum_{i}\lambda_{i}(dv_k)=\sum_{i}{\partial F(z)\over \partial x_s}=-{m\over n}zF(z)\sum_{i}{1\over 1-zx_i}=$$ $$-mzF(z)-{m\over n}z^{2}F(z)\sum_{i}{x_i\over 1-zx_i}=-mzF(z)+z^2{dF(z)\over dz}.$$
\end{proof}

\begin{theorem}
\label{Huv}
The action of Hamiltonians on $u_l$ and $v_l$ has the following form:
\begin{eqnarray}
\label{Hk}
H_{k}(v_l)=(l-1-m)\cdots (l-k+m)v_{l-k},\nonumber \\ H_{k}(u_l)=({m\over n})^{k-1}(l-1-n)\cdots(l-k-n)u_{l-k}
\end{eqnarray}
\end{theorem}

\begin{proof}
Since $u_l$ is symmetric,
$D_{i}(u_{l})={\partial u_{l}\over \partial x_{i}},$
and one can check that 
$$D_{i}^2(u_{l})={m\over n}(l-1-n){\partial u_{l-1}\over \partial x_{i}},\quad \sum_{i}\lambda_i(du_l)=(l-1-n)u_{l-1}$$
hence
$$H_{k}(u_l)=({m\over n})^{k-1}(l-1-n)\cdots(l-k+1-n)\sum_{i}{\partial u_{l-k+1}\over \partial x_{i}}=$$ $$
({m\over n})^{k-1}(l-1-n)\cdots(l-k+1-n)(l-k-n)u_{l-k}.$$

The proof for $v_l$ is similar -- it follows from  Lemma \ref{Dv} and Lemma \ref{sumlam}.
\end{proof}

\bigskip

Department of Mathematics, Stony Brook University.

100 Nicolls Road, Stony Brook NY 11733.

{\it E-mail address:} {\tt egorsky@math.sunysb.edu}

\end{document}